\documentclass[11pt]{article}

\usepackage{comment}
\usepackage{amsmath,amssymb,amsthm,comment}
\usepackage{pstricks,pstricks-add,pst-eucl,pst-func,pst-plot}
\usepackage[colorlinks,citecolor=red,urlcolor=blue,linkcolor=blue]{hyperref}
\usepackage{mathrsfs,dsfont}

\newtheorem{nummer}{ }
\newtheorem{thm}[nummer]{\sc Theorem}

\newtheorem{lem}[nummer]{\sc Lemma}
\newtheorem{cor}[nummer]{\sc Corollary}

\newtheorem{algo}[nummer]{\sc Fermat's Algorithm}

\newcommand{\ie} {{\sl i.e.}}

\newcommand{\dritt}{\#}

\newcommand{\Q}{\mathds{Q}}

\newcommand{\Z}{\mathds{Z}}
\newcommand{\R}{\mathds{R}}

\newcommand{\THM}{\sc Theorem}
\newcommand{\FA}{\sc Fermat's Algorithm}

\newcommand{\LEM}{\sc Lemma}

\newcommand{\nO}{\mathscr{O}}

\newcommand{\tu}{\tilde u}
\newcommand{\tv}{\tilde v}

\newcommand{\tA}{\tilde A}

\newcommand{\tripleAbf}{\bf rational Pythagorean $\boldsymbol{A}$-triple}
\newcommand{\tripleA}{rational Py\-tha\-go\-re\-an $A$-triple}

\makeatletter
\def\opargproof[#1]{\par\noindent {\bf #1 }}
 \makeatother

\begin{document}
\begin{center}
{\LARGE\bf A Theorem of Fermat}\\[1.7ex] 
{\Large\bf on}\\[1.7ex] 
{\LARGE\bf Congruent Number Curves}

\medskip

{\small Lorenz Halbeisen}\\[1.2ex] 
{\scriptsize Department of Mathematics, ETH Zentrum,
R\"amistrasse\;101, 8092 Z\"urich, Switzerland\\ lorenz.halbeisen@math.ethz.ch}\\[1.8ex]
{\small Norbert Hungerb\"uhler}\\[1.2ex] 
{\scriptsize Department of Mathematics, ETH Zentrum,
R\"amistrasse\;101, 8092 Z\"urich, Switzerland\\ norbert.hungerbuehler@math.ethz.ch}
\end{center}

\hspace{5ex}{\small{\it key-words\/}: congruent numbers, Pythagorean triples}

\hspace{5ex}{\small{\it 2010 Mathematics Subject 
Classification\/}: {\bf 11G05}\,\ 11D25}

\begin{abstract}\noindent
A positive integer $A$ is called a {\it congruent number\/} if 
$A$ is the area of a right-angled triangle with three rational 
sides. Equivalently, $A$ is a {\it congruent number\/} 
if and only if the congruent number curve $y^2=x^3-A^2x$ has
a rational point $(x,y)\in\Q^2$ with $y\neq 0$.
Using a theorem of Fermat, we give an elementary proof for
the fact that congruent number curves do not contain rational 
points of finite order. 
\end{abstract}

\section{Introduction}

A positive integer $A$ is called a {\bf congruent number} if 
$A$ is the area of a right-angled triangle with three rational 
sides. So, $A$ is congruent if and only if there exists a
rational Pythagorean tripel $(a,b,c)$ (\ie, $a,b,c\in\Q$,
$a^2+b^2=c^2$, and $ab\neq 0$), such that $\frac{ab}2=A$.
The sequence of integer congruent numbers starts with
$$
5, 6, 7, 13, 14, 15, 20, 21, 22, 23, 24, 28, 29, 30, 31, 34, 37,\ldots
$$
For example, $A=7$ is a congruent number, 
witnessed by the rational Pythagorean triple $$\Bigl(\frac{24}{5}\,,
\frac{35}{12}\,,\frac{337}{60}\Bigr).$$

It is well-known that $A$ is a congruent number if
and only if the cubic curve $$C_A:\ y^2=x^3-A^2 x$$
has a rational point $(x_0,y_0)$ with $y_0\neq 0$.
The cubic curve $C_A$ is called a {\bf congruent number curve}.
This correspondence between rational points on congruent number curves and 
rational Pythagorean triples can be made explicit as follows:
Let 
$$
C(\Q):= \{(x,y,A)\in \Q\times\Q^*\times \Z^*:y^2=x^3-A^2x\},
$$
where $\Q^*:=\Q\setminus\{0\}, \Z^*:=\Z\setminus\{0\}$, and
$$
P(\Q):=\{(a,b,c,A)\in \Q^3\times\Z^*:a^2+b^2=c^2\ \textsl{and\/}\ ab=2A\}.
$$
Then, it is easy to check that
\begin{equation}\label{psi}
\begin{aligned}
\psi\ :\ \quad P(\Q)&\ \to\  C(\Q)\\
(a,b,c,A)&\ \mapsto \ \Bigl(\frac{A(b+c)}{a}\,,\,\frac{2A^2(b+c)}{a^2}\,,\,A\Bigr)
\end{aligned}
\end{equation}
is bijective  and
\begin{equation}\label{psi-1}
\begin{aligned}
\psi^{-1}\ :\qquad C(\Q)&\ \to\  P(\Q)\\
        (x,y,A)&\ \mapsto\  \Bigl(\frac{2x A}{y}\,,\;\frac{x^2-A^2}{y}\,,\;\frac{x^2+A^2}{y}\,,\,A\Bigr).
\end{aligned}
\end{equation}

For positive integers $A$, a triple $(a,b,c)$ of 
non-zero rational numbers is called a {\tripleAbf} 
if $a^2+b^2=c^2$ and $A=\big{|}\frac{ab}{2}\big{|}$.
Notice that if $(a,b,c)$ is a {\tripleA}, then $A$ 
is a congruent number and $|a|,|b|,|c|$ are the 
lengths of the sides of a right-angled triangle 
with area $A$. Notice also that we allow $a,b,c$
to be negative.

It is convenient to consider the curve $C_A$ in the
projective plane $\R P^2$, where the curve is given by
$$
C_A :\ y^2z = x^3-A^2xz^2.
$$
On the points of $C_A$, one can define a commutative, binary, 
associative operation ``$+$'', where $\nO$, the neutral 
element of the operation, is the projective point $(0,1,0)$
at infinity. More formally, if $P$ and $Q$ are two points on $C_A$, 
then let $P\dritt Q$ be the third intersection point of
the line through $P$ and $Q$ with the curve $C_A$. 
If $P=Q$, the line through $P$ and $Q$ is replaced by the tangent in $P$.
Then
$P+Q$ is defined by stipulating 
$$P+Q\;:=\;\nO\dritt (P\dritt Q),$$
where for a point $R$ on $C_A$, $\nO\dritt R$ is the point reflected across the $x$-axis.
The following figure shows the congruent number curve $C_A$ for
$A=5$, together with two points $P$ and $Q$ and their sum $P+Q$.
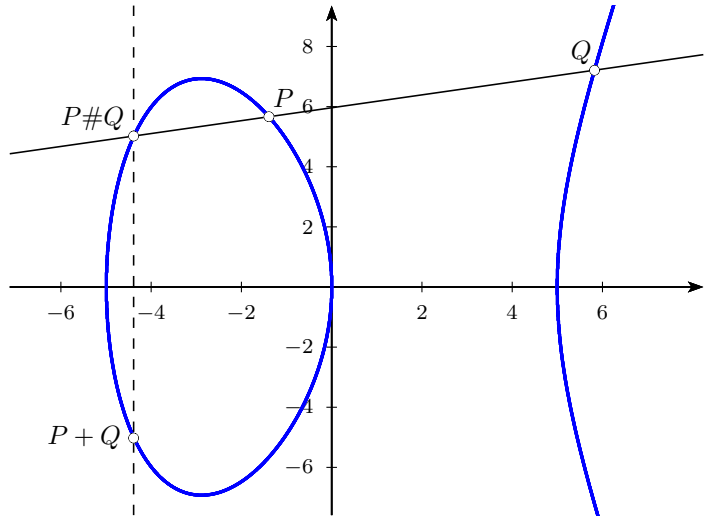
\begin{center}
\psset{xunit=.6cm,yunit=.4cm,algebraic=true,dimen=middle,dotstyle=o,dotsize=5pt 0,linewidth=1.6pt,arrowsize=3pt 2,arrowinset=0.25}
\begin{pspicture*}(-7.1329967371431815,-7.588280903625012)(8.234736979744335,9.360344080789211)
\psaxes[labelFontSize=\scriptstyle,xAxis=true,yAxis=true,Dx=2.,Dy=2.,ticksize=-2pt 2pt,subticks=1,linewidth=.6pt,]{->}(0,0)(-7.1329967371431815,-7.588280903625012)(8.234736979744335,9.360344080789211)
\psplotImp[linewidth=1.2pt,linecolor=blue,stepFactor=0.1](-9.0,-9.0)(9.0,10.0){1.0*y^2+25.0*x^1-1.0*x^3}
\psplot[linewidth=.6pt]{-7.1329967371431815}{8.234736979744335}{(-43.01566031988557-1.5476943995438228*x)/-7.215284285929719}
\psline[linewidth=.6pt,linestyle=dashed,dash=4pt 4pt](-4.388182328551535,-7.588280903625012)(-4.388182328551535,9.360344080789211)
\begin{small}
\psdots[dotsize=4pt 0](5.824738905621608,7.211160924647724)
\rput[bl](5.3,7.4){$Q$}
\psdots[dotsize=4pt 0](-1.3905453803081107,5.663466525103901)
\rput[bl](-1.3,6){$P$}
\psdots[dotsize=4pt 0](-4.388182328551535,5.020466785549749)
\rput[bl](-6.,5.2){$P\dritt Q$}
\psdots[dotsize=4pt 0](-4.388182328551535,-5.020466785549749)
\rput[bl](-6.3,-5.4){$P+Q$}
\end{small}
\end{pspicture*}
\end{center}
More formally, for two points $P=(x_0,y_0)$ and $Q=(x_1,y_1)$ on
a congruent number curve $C_A$, the point $P+Q=(x_2,y_2)$ is given by
the following formulas: 
\begin{itemize}
\item If $x_0\neq x_1$, then 
$$x_2=\lambda^2-x_0-x_1,\qquad y_2=\lambda(x_0-x_2)-y_0,$$
where $$\lambda:=\frac{y_1-y_0}{x_1-x_0}.$$
\item If $P=Q$, \ie, $x_0=x_1$ and $y_0=y_1$, then 
\begin{equation}\label{eq:2P}
x_2=\lambda^2-2x_0,\qquad y_2=3x_0\lambda-\lambda^3-y_0,
\end{equation}
where 
\begin{equation}\label{eq:lambda}
\lambda:=\frac{3x_0^2-A^2}{2y_0}.
\end{equation}
Below we shall write $2*P$ instead of $P+P$. 

\item If $x_0=x_1$ and $y_0=-y_1$, then $P+Q:=\nO$. In particular,
$(0,0)+(0,0)=(A,0)+(A,0)=(-A,0)+(-A,0)=\nO$. 
\item Finally,
we define $\nO+P:=P$ and $P+\nO:=P$ for any point $P$, in particular, 
$\nO+\nO=\nO$.
\end{itemize}
With the operation~``$+$'',
$(C_A,+)$ is an abelian group with neutral element $\nO$.
Let $C_A(\Q)$ be the set of rational points on $C_A$ together
with $\nO$. It is easy to see that $\bigl(C_A(\Q),+\bigr)$. 
is a subgroup of $(C_A,+)$. Moreover, it is well known that 
the group $\bigl(C_A(\Q),+\bigr)$ is finitely generated.
One can readily check that the three points $(0,0)$ and
$(\pm A,0)$ are the only points on $C_A$ of order~$2$,
and one easily finds other points of finite order on $C_A$.
But do we find also rational points of finite order on~$C_A$?
This question is answered by the following 

\begin{thm}\label{thm:main}
If $A$ is a congruent number and $(x_0,y_0)$ is a rational
point on $C_A$ with $y_0\neq 0$, then the order of $(x_0,y_0)$
is infinite. In particular, if there exists one {\tripleA}, then
there exist infinitely many such triples.
\end{thm}

The usual proofs of {\THM}\;\ref{thm:main} are quite involved. For example,
Koblitz~\cite[Ch.\,I, \S\,9, Prop.\,17]{Koblitz} gives a proof using
Dirichlet's theorem on primes in an arithmetic progression, and in
Chahal~\cite[Thm.\,3]{Chahal}, a proof is given using the Lutz-Nagell theorem,
which states that rational points of finite order are integral.
However, both results, Dirichlet's theorem and the Lutz-Nagell theorem, are 
quite deep results, and the aim of this article is to provide a simple proof
of {\THM}\;\ref{thm:main} which relies on an elementary theorem of Fermat.

\section{A Theorem of Fermat}

In~\cite{Fermat}, Fermat gives an algorithm to construct 
different right-angled triangles with three rational 
sides having the same area (see also Hungerb\"uhler~\cite{Noebi}).
Moreover, Fermat claims that his 
algorithm yields infinitely many distinct such right-angled 
triangles. However, he did not provide a proof for this 
claim. In this section, we first present Fermat's algorithm
and then we show that this algorithm delivers infinitely many
pairwise distinct rational right-angled triangles of the
same area.

\begin{algo}\label{algo}
Assume that $A$ is a congruent number, and
that $(a_0,b_0,c_0)$ is a {\tripleA},
\ie, $A=\big{|}\frac{a_0 b_0}2\big{|}$.  
Then 
\begin{equation}\label{eq:fermat}
a_1:=\frac{4c_0^2a_0b_0}{2c_0(a_0^2-b_0^2)},\quad
b_1:=\frac{c_0^4-4a_0^2b_0^2}{2c_0(a_0^2-b_0^2)},\quad
c_1:=\frac{c_0^4+4a_0^2b_0^2}{2c_0(a_0^2-b_0^2)},
\end{equation}
is also a {\tripleA}. Moreover, $a_0b_0=a_1b_1$, \ie, if $(a_0,b_0,c_0,A)\in P(\Q)$,
then $(a_1,b_1,c_1,A)\in P(\Q)$.
\end{algo}

\begin{proof}
Let $m:=c_0^2$, let $n:=2a_0b_0$, and let 
$$X:=2 m n,\quad Y:=m^2-n^2,\quad Z:=m^2+n^2,$$
in other words,
$$X=4c_0^2a_0b_0,\quad
Y=c_0^4-4a_0^2b_0^2,\quad
Z=c_0^4+4a_0^2b_0^2.$$
Then obviously, $X^2+Y^2=Z^2$, and since $a_0,b_0,c_0\in\Q$,
$\bigl(|X|,|Y|,|Z|\bigr)$ is a rational Pythagorean triple, where the
area of the corresponding right-angled triangle
is $$\tA=\bigg{|}\frac{X Y}{2}\bigg{|}=
\big{|}2 a_0 b_0 c_0^2 (c_0^4-4 a_0^2 b_0^2)\big{|}.$$ 
Since $a_0^2+b_0^2=c_0^2$, we get $c_0^4=(a_0^2+b_0^2)^2=a_0^4+2a_0^2b_0^2+b_0^4$
and therefore $$c_0^4-4 a_0^2 b_0^2\;=\;a_0^4-2a_0^2b_0^2+b_0^4\;=\;(a_0^2-b_0^2)^2>0.$$
So, for $$a_1=\frac{X}{2c_0(a_0^2-b_0^2)},\quad
b_1=\frac{Y}{2c_0(a_0^2-b_0^2)},\quad
c_1=\frac{Z}{2c_0(a_0^2-b_0^2)},$$
we have $a_1^2+b_1^2=c_1^2$ and
$$\frac{a_1b_1}{2}\;=\;
\frac{XY}{2\cdot 4c_0^2(a_0^2-b_0^2)^2}\;=\;
\frac{2 a_0 b_0 c_0^2 (c_0^4-4 a_0^2 b_0^2)}{4c_0^2(a_0^2-b_0^2)^2}\;=\;
\frac{2 a_0 b_0 c_0^2 (a_0^2-b_0^2)^2}{4c_0^2(a_0^2-b_0^2)^2}\;=\;
\frac{a_0 b_0}{2}.
$$
\end{proof}

\begin{thm}\label{thm:FermatClaim} 
Assume that $A$ is a congruent number,
that $(a_0,b_0,c_0)$ is a {\tripleA},
and for positive integers~$n$, let
$(a_n,b_n,c_n)$ be the {\tripleA}
we obtain by {\FA} from $(a_{n-1},b_{n-1},c_{n-1})$.
Then for any distinct non-negative integers $n,n'$, we have
$|c_n|\neq |c_{n'}|$. 
\end{thm}

\begin{proof} Let $n$ be an arbitrary but fixed
non-negative integer. Since $A=\big{|}\frac{a_n b_n}2\big{|}$, 
we have $2A=|a_nb_n|$, and consequently 
\begin{equation}\label{*}
a_n^2b_n^2=4A^2.
\end{equation} 
Furthermore, 
since $a_n^2+b_n^2=c_n^2$, we have 
$$(a_n^2+b_n^2)^2=a_n^4+2a_n^2b_n^2+b_n^4=a_n^4+8A^2+b_n^4=c_n^4,$$
and consequently we get
$$c_n^4-16A^2=a_n^4-8A^2+b_n^4=a_n^4-2a_n^2b_n^2+b_n^4=(a_n^2-b_n^2)^2>0.$$
Therefore, $$\sqrt{(a_n^2-b_n^2)^2}=|a_n^2-b_n^2|=\sqrt{c_n^4-16A^2},$$
and with~(\ref{eq:fermat}) and~(\ref{*}) we finally have 
$$|c_{n+1}|=\frac{c_n^4+16A^2}{2c_n\sqrt{c_n^4-16A^{2\mathstrut}}}\,.$$
Now, assume that $c_n=\frac uv$ where $u$ and $v$ are in lowest terms.
We consider the following two cases:

{\it $u$ is odd\/}: First, we write $v=2^k\cdot\tv$, where $k\ge 0$
and $\tv$ is odd. In particular, $c_n=\frac{u}{2^{k\mathstrut}\cdot\tv}$. 
Since $c_{n+1}$ is rational, $\sqrt{c_n^4-16A^2}\in\Q$.
So, $$\sqrt{c_n^4-16A^2}=\sqrt{\frac{u^4-16A^2v^4}{v^4}}=\frac{\tu}{v^2}$$
for a positive odd integer $\tu$. Then 
$$|c_{n+1}|=\frac{\frac{u^4+16A^2v^4}{v^4}}{\frac{2u\tu}{v^3}}=
\frac{\bar{u}}{2u\tu v}=
\frac{\bar{u}}{2u\tu 2^{k\mathstrut}\tv}=\frac{\bar{u}}{2^{k+1}u\tu\tv}
=\frac{u'}{2^{k+1\mathstrut}\cdot v'}$$
where $\bar{u},u',v'$ are odd integers and $\gcd(u',v')=1$. This shows
that $$c_n=\frac{u}{2^{k\mathstrut}\cdot\tv}\quad\Rightarrow\quad 
|c_{n+1}|=\frac{u'}{2^{k+1\mathstrut}\cdot v'}$$ where $u,\tv,u',v'$ are odd.

{\it $u$ is even\/}: First, we write $u=2^k\cdot\tu$, where $k\ge 1$
and $\tu$ is odd. In particular, $c_n=\frac{2^k\cdot\tu}{v}$, where $v$ is odd.
Similarly, $A=2^l\cdot\tA$, where $l\ge 0$ and $\tA$ is odd.
Then $$c_n^4\pm 16A^2=\frac{2^{4k}\cdot\tu^4\pm 2^{4+2l}\tA^2 v^4}{v^4},$$ 
where both numbers are of the form $$\frac{2^{2m}\bar{u}}{v^4}\,,$$
where $\bar{u}$ is odd and $4\le 2m\le 4k$, \ie, $2\le m\le 2k$. 
Therefore, $$|c_{n+1}|=\frac{2^{2m}u_0\cdot v^3}{2\cdot 2^k\tu\cdot
v^4\cdot2^{m\mathstrut}u_1}=\frac{2^{m-k-1\mathstrut}\cdot u'}{v'},$$
where $u_0,u_1,u',v'$ are odd. Since $m<2k+1$, we have $m-k-1<k$, and
therefore we obtain
$$c_n=\frac{2^k\cdot\tu}{v}\quad\Rightarrow\quad 
|c_{n+1}|=\frac{2^{k'}\cdot u'}{v'}$$ where $\tu,v,u',v'$ are odd and
$0\le k'<k$. 

Both cases together show that whenever $c_n=2^k\cdot\frac uv$, where
$k\in\Z$ and $u,v$ are odd, then $|c_{n+1}|=2^{k'}\cdot\frac{u'}{v'}$,
where $u',v'$ are odd and $k'<k$. So, for any distinct non-negative 
integers $n$ and $n'$, $|c_n|\neq |c_{n+1}|$. 
\end{proof}

The proof of {\THM}\;\ref{thm:FermatClaim} gives us the following
reformulation of {\FA}:

\begin{cor}\label{cor:FA}
Assume that $A$ is a congruent number, and
that $(a_0,b_0,c_0)$ is a {\tripleA},
\ie, $A=\big{|}\frac{a_0 b_0}2\big{|}$.  
Then 
$$a_1=\frac{4Ac_0}{\sqrt{c_0^4-16A^2}},\quad
b_1=\frac{\sqrt{c_0^4-16A^2}}{2c_0},\quad
c_1=\frac{c_0^4+16A^2}{2c_0\sqrt{c_0^4-16A^2}},
$$
is also a {\tripleA}.
\end{cor}

\begin{proof} Notice that $c_0^4-4a_0^2b_0^2=c_0^4-16A^2$ 
and recall that $|a_0^2-b_0^2|=\sqrt{c_0^4-16A^2}$.
\end{proof}

\section{Doubling points with Fermat's Algorithm}

Before we prove {\THM}\;\ref{thm:main} (\ie,
that congruent number curves do not 
contain rational points of finite order), we first prove
that {\FA}\;\ref{algo} is essentially doubling points
on congruent number curves.

\begin{lem}\label{lem:doubling}
Let $A$ be a congruent number, let $(a_0,b_0,c_0)$ 
be a {\tripleA}, and let $(a_1,b_1,c_1)$ be
the {\tripleA} obtained by {\FA} from $(a_0,b_0,c_0)$.
Furthermore, let $(x_0,y_0)$ and $(x_1,y_1)$ be the rational
points on the curve $C_A$ which correspond to 
$(a_0,b_0,c_0)$ and $(a_1,b_1,c_1)$, respectively.
Then we have 
$$
2*(x_0,y_0)=(x_1,-y_1).$$
\end{lem}

\begin{proof}

Let $(a_0,b_0,c_0)$ be a {\tripleA}. Then, according to~(\ref{eq:fermat}), the 
{\tripleA} $(a_1,b_1,c_1)$ which we obtain by {\FA} is given
by $$a_1:=\frac{4c_0^2a_0b_0}{2c_0(a_0^2-b_0^2)},\quad
b_1:=\frac{c_0^4-4a_0^2b_0^2}{2c_0(a_0^2-b_0^2)},\quad
c_1:=\frac{c_0^4+4a_0^2b_0^2}{2c_0(a_0^2-b_0^2)}.
$$ 
Now, by~(\ref{psi}), the coordinates of the rational point 
$(x_1,y_1)$ on $C_A$ which corresponds to the {\tripleA}
$(a_1,b_1,c_1)$ are given by 
\begin{align*}
 x_1&=
\frac{a_0b_0\cdot(b_1+c_1)}{2\cdot a_1}=\frac{a_0b_0\cdot 2c_0^4}{2\cdot 4c_0^2a_0b_0}=
\frac{c_0^2}{4}\,,\\
y_1&=\frac{2(\frac{a_0b_0}2)^2(b_1+c_1)}{a_1^2}=\frac18 (a_0^2 - b_0^2) c_0.
\end{align*}
Let still $(a_0,b_0,c_0)$ be a {\tripleA}. Then, again by~(\ref{psi}), the corresponding
rational point $(x_0,y_0)$ on $C_A$ is given by 
$$x_0=\frac{b_0(b_0+c_0)}{2}\,,\qquad
y_0=\frac{b_0^2(b_0+c_0)}{2}\,.$$
Now, as we have seen in~(\ref{eq:2P}) and~(\ref{eq:lambda}),
the coordinates of the point $(x'_1,y'_1):=2*(x_0,y_0)$ are given by
$x'_1=\lambda^2-2x_0$, $y'_1=3x_0\lambda-\lambda^3-y_0$, where 
\begin{multline*}
 \lambda=\frac{3x_0^2-(\frac{a_0b_0}2)^2}{2y_0}=
 \frac{\frac{3 b_0^2(b_0+c_0)^2-a_0^2b_0^2}{4}}{b_0^2(b_0+c_0)}=
 \frac{3 (b_0+c_0)^2-a_0^2}{4(b_0+c_0)}=
 \frac{3 (b_0+c_0)^2+(b_0^2-c_0^2)}{4(b_0+c_0)}=\\[3ex]
 \frac{(3b_0^2+6b_0c_0+3c_0^2)+(b_0^2-c_0^2)}{4(b_0+c_0)}=
 \frac{4b_0^2+6b_0c_0+2c_0^2}{4(b_0+c_0)}=
 \frac{2b_0^2+3b_0c_0+c_0^2}{2(b_0+c_0)}=\\[3ex]
 \frac{(2b_0+c_0)(b_0+c_0)}{2(b_0+c_0)}=
 \frac{(2b_0+c_0)}{2}\,.
\end{multline*}
Hence, 
$$x_1'=\lambda^2-2x_0=\frac{(2b_0+c_0)^2}{4}-b_0(b_0+c_0)=
\frac{(4b_0^2+4b_0c_0+c_0^2)-(4b_0^2+4b_0c_0)}{4}=\frac{c_0^2}{4}\,
$$
and
$$
y'_1=3x_0\lambda-\lambda^3-y_0=\frac18 (2 b_0^2 c_0 - c_0^3)=\frac18(b_0^2-a_0^2)c_0,
$$
\ie, $x_1=x'_1$ and $y_1=-y'_1$, as claimed.
\end{proof}

With {\LEM}\;\ref{lem:doubling}, we are now able to
prove {\THM}\;\ref{thm:main}, which states that
for a congruent number $A$, the curve 
$C_A:y^2=x^3-A^2x$ does not have rational points of finite 
order other than $(0,0)$ and $(\pm A,0)$. 

\begin{proof}[Proof of Theorem\;\ref{thm:main}]
Assume that $A$ is a congruent number,
let $(x_0,y_0)$ be a rational point on $C_A$ which
$y_0\neq 0$, and let $(a_0,b_0,c_0)$ be the {\tripleA}
which corresponds to $(x_0,y_0)$ by~(\ref{psi-1}).
Furthermore, for positive integers~$n$, let
$(a_n,b_n,c_n)$ be the {\tripleA}
we obtain by {\FA} from $(a_{n-1},b_{n-1},c_{n-1})$,
and let $(x_n,y_n)$ be the 
rational point on $C_A$ which corresponds to the 
{\tripleA} $(a_n,b_n,c_n)$ by~(\ref{psi}). 

By the proof of {\LEM}\;\ref{lem:doubling} we know that the
$x$-coordinate of $2*(x_n,y_n)$ is equal to $\frac{c_n^2}{4}$,
and by {\THM}\;\ref{thm:FermatClaim} we have that for any distinct 
non-negative integers $n,n'$, $|c_n|\neq |c_{n'}|$. 
Hence, for all distinct 
non-negative integers $n,n'$ we have 
$$(x_n,y_n)\neq (x_{n'},y_{n'}),$$ which shows that the order
of $(x_0,y_0)$ is infinite.
\end{proof}

\providecommand{\bysame}{\leavevmode\hbox to3em{\hrulefill}\thinspace}
\providecommand{\MR}{\relax\ifhmode\unskip\space\fi MR }
\providecommand{\MRhref}[2]{%
  \href{http://www.ams.org/mathscinet-getitem?mr=#1}{#2}
}
\providecommand{\href}[2]{#2}

\end{document}